\documentclass[12pt,a4paper]{article}
\title{Correspondence of the eigenvalues of a non-self-adjoint operator to those of a self-adjoint operator}
\author{John Weir\\
				Department of Mathematics,
				King's College London,
				Strand,\\
				London WC2R 2LS,
				United Kingdom\\
				john.l.weir@kcl.ac.uk}

\newtheorem{theorem}{Theorem}[section]
\newtheorem{lemma}[theorem]{Lemma}
\newtheorem{proposition}[theorem]{Proposition}
\newtheorem{corollary}[theorem]{Corollary}

\newenvironment{proof}[1][Proof]{\begin{trivlist}
\item[\hskip \labelsep {\bfseries #1}]}{\end{trivlist}}

\newcommand{\qed}{\nobreak \ifvmode \relax \else
      \ifdim\lastskip<1.5em \hskip-\lastskip
      \hskip1.5em plus0em minus0.5em \fi \nobreak
      \vrule height0.75em width0.5em depth0.25em\fi}

\newcommand{\setN}{\mathbf{N}}
\newcommand{\setR}{\mathbf{R}}
\newcommand{\setZ}{\mathbf{Z}}
\newcommand{\setC}{\mathbf{C}}
\newcommand{\eqnref}[1]{(\ref{#1})}
\newcommand{\norm}[1]{\left|\left|#1\right|\right|}
\newcommand{\mod}[1]{\left| #1 \right|}
\newcommand{\dom}[1]{\mathrm{Dom}\left(#1\right)}

\newcommand{\conj}[1]{\overline{#1}}

\newcommand{\ip}[2]{\left< #1, #2 \right>}
\newcommand{\supp}[1]{\mathrm{Supp}\left( #1 \right)}
\newcommand{\dd}[1]{\frac{\mathrm{d}}{\mathrm{d} #1}}
\newcommand{\pdd}[1]{\frac{\partial}{\partial #1}}
\newcommand{\spec}[1]{\mathrm{Spec}\left(#1\right)}

\newcommand{\Ltwow}{\mathrm{L}^2((0,1),w(x)\mathrm{d}x)}
\newcommand{\Ltwo}{\mathrm{L}^2(0,1)}
\newcommand{\Lone}{\mathrm{L}^1(0,1)}
\newcommand{\smooth}{\mathcal{C}^{\infty}([0,1])}
\newcommand{\tnorm}[1]{\left|\left|\left|#1\right|\right|\right|}

\begin{document}
\maketitle

\begin{abstract}
We prove that the eigenvalues of a certain highly non-self-adjoint operator that arises in fluid mechanics correspond, up to scaling by a positive constant, to those of a self-adjoint operator with compact resolvent; hence there are infinitely many real eigenvalues which accumulate only at $\pm \infty$. We use this result to determine the asymptotic distribution of the eigenvalues and to compute some of the eigenvalues numerically. We compare these to earlier calculations in \cite{bobs}, \cite{davies-2007} and \cite{chugunova-2007}.

MSC classes: 34Lxx; 76Rxx; 34B24

Keywords: spectrum, non-self-adjoint, self-adjoint, fluid mechanics, eigenvalue, Sturm-Liouville
\end{abstract}

\section{Introduction}\label{sec:intro}

In a recent paper \cite{weir-2007}, we showed that the spectrum of the highly non-self-adjoint operator $-iH$ is real, where $H$ is the closure of the operator $H_0$ on $L^2(-\pi,\pi)$ defined by
\begin{equation}\label{eq:h0}
	(H_0f)(\theta) = \varepsilon\pdd{\theta}\left(\sin(\theta)\frac{\partial f}{\partial \theta}\right) + \frac{\partial f}{\partial \theta}
	\end{equation}
for any fixed $\varepsilon \in (0,2)$ and all $f \in \dom{H_0} = \mathcal{C}^2_{\mathrm{per}}([-\pi,\pi])$. Boulton, Levitin and Marletta subsequently proved in a recent paper \cite{boulton-2008} that a wider class of operators possess only real eigenvalues. However, they did not prove that any eigenvalues exist for these operators, nor that their spectra are real. The results obtained in this paper for the original operator \eqnref{eq:h0} are much more detailed than those presented in \cite{weir-2007, boulton-2008}.

The operator $H$ was first studied by Benilov, O'Brien and Sazonov, who argued in \cite{bobs} that the equation
\begin{equation}
	\frac{\partial f}{\partial t} = Hf
	\end{equation}
approximates the evolution of a liquid film inside a rotating horizontal cylinder. They also made several conjectures, based on non-rigorous numerical analysis, including that the spectrum of $H$ is purely imaginary and consists of eigenvalues which accumulate at $\pm i \infty$.

Davies showed in \cite{davies-2007} that $-iH$ has compact resolvent by considering the unitarily equivalent operator $A$ on $l^2(\setZ)$ defined by
\begin{equation}
	(Av)_n = \frac{\varepsilon}{2}n(n-1)v_{n-1} - \frac{\varepsilon}{2}n(n+1)v_{n+1}+nv_n
	\end{equation}
for all $v \in \dom{A} = \{v \in l^2(\setZ) : Av \in l^2(\setZ) \}$. Here $A = \mathcal{F}^{-1}(-iH)\mathcal{F}$, where $\mathcal{F}:L^2(-\pi,\pi) \to l^2(\setZ)$ is the Fourier transform. If $\mathcal{F}f=v$ then $(v_n)_{n \in \setZ}$ are the Fourier coefficients of $f$. This result was achieved by obtaining sharp bounds on the rate of decay of eigenvectors and resolvent kernels, and by determining the precise domains of the operators involved.
 He also showed that
\begin{equation}\label{eq:a}
A = A_- \oplus 0 \oplus A_+,
\end{equation}
where $A_-$ and $A_+$ are the restrictions of $A$ to $l^2(\setZ_-)$ and $l^2(\setZ_+)$ respectively, and that $A_-$ is unitarily equivalent to $-A_+$. Since the resolvent is compact and the adjoint has the same eigenvalues, the spectrum of $-iH$ consists entirely of eigenvalues.

As previously mentioned, we proved in \cite{weir-2007} that these eigenvalues, if they exist, must all be real. Eigenvalues of $H$ or $-iH$ have been calculated numerically in \cite{bobs,davies-2007,chugunova-2007}, but until now it has not been proven rigorously that any non-zero eigenvalues exist.

In this paper we prove rigorously that $-iH$ has infinitely many eigenvalues which accumulate at $\pm \infty$ (Corollary \ref{cor:hevs}). Our approach is to show that the eigenvalues of $A_+$ correspond, up to scaling by a positive constant, to those of a self-adjoint operator with compact resolvent (Corollary \ref{cor:correspondence}, Theorem \ref{thm:sa}, Theorem \ref{thm:inverse}) . By analysing the self-adjoint operator, we determine the asymptotic distribution of the eigenvalues (Theorem \ref{thm:muasym}). It was argued in \cite{chugunova-2007} that the distribution of the eigenvalues, if they exist, should be quadratic, but no rigorous bounds were given. We prove rigorously that $\lambda_n \sim \varepsilon \pi^2 n^2 \beta^{-2}$ for some constant $\beta$ which we determine. We also perform numerical calculations of eigenvalues, which we compare to those given in \cite{bobs}, \cite{davies-2007} and \cite{chugunova-2007} (Section \ref{sec:numerics}). Moreover, our calculated values are rigorous upper bounds on the true values of the eigenvalues, insofar as the computed eigenvalues of regular Sturm-Liouville problems, which are known to be computationally stable, can be said to be rigorous. This gives us some idea of the accuracy of the previous calculations.

The correspondence of the eigenvalues of $iH$ to those of a self-adjoint operator $Q$ might lead us to believe that $iH$ is similar to $Q$ in the sense that there exists a bounded linear operator $S$ with bounded inverse such that $S \dom{Q} = \dom{H}$ and $iH = SQS^{-1}$. However, it has recently been proven that $iH$ is not similar in this sense to any self-adjoint operator \cite[Proof of Theorem 5.1]{chugunova-2008}.

\section{Correspondence of eigenvalues to those of a Sturm-Liouville problem}\label{sec:correspondence}

We have already shown in \cite{weir-2007} that if $\lambda$ is an eigenvalue of the operator $A_+$ defined on its natural maximal domain by
\[(A_+v)_n = \frac{\varepsilon}{2}n(n-1)v_{n-1} - \frac{\varepsilon}{2}n(n+1)v_{n+1} + nv_n \]
then $\mu = 2\lambda/\varepsilon$ is an eigenvalue of the Sturm-Liouville problem
\begin{equation}
	-(pu')' = \mu w u,
	\label{eq:sl}
\end{equation}
where
\begin{eqnarray}
	p(x) & = & (1-x)^{1+1/\varepsilon}(x+1)^{1-1/\varepsilon},\\
	w(x) & = & x^{-1}(1-x)^{1/\varepsilon}(x+1)^{-1/\varepsilon}
\end{eqnarray}
and $u \in \mathcal{C}^{\infty}([0,1])$ with $u(0) = 0$. Moreover, the solution of \eqnref{eq:sl} satisfying these conditions is
\begin{equation}\label{eq:udef}
u(x) = \sum_{n=1}^{\infty}v_nx^n,
\end{equation}
where $v_n$ is the solution of the recurrence relation
\begin{equation}
	n(n-1)v_{n-1}-n(n+1)v_{n+1}+2\frac{n-\lambda}{\varepsilon}v_n=0
	\label{eq:recurrence}
\end{equation}
satisfying the initial conditions $v_1 = 1$, $v_2 = (1-\lambda)/\varepsilon$.

We now show the converse:

\begin{theorem}
If $\mu$ is an eigenvalue of the Sturm-Liouville problem \eqnref{eq:sl}, then $\lambda = \varepsilon \mu/2$ is an eigenvalue of $A_+$.
\end{theorem}
\begin{proof}
If $(v_n)$ is the solution of the recurrence relation \eqnref{eq:recurrence} satisfying the stated initial conditions and $u$ is defined by \eqnref{eq:udef} on $(0,1)$, then $u$ is a non-zero solution of \eqnref{eq:sl}.
Equation \eqnref{eq:sl} is equivalent to
\begin{equation}
	u'' + \left(\frac{1+1/\varepsilon}{z-1}+\frac{1-1/\varepsilon}{z+1}\right)u'-\frac{\mu}{z(z-1)(z+1)}u=0,
	\label{eq:heun}
\end{equation}
so we see that a second linearly independent solution is $u_1 = au(z)\log{z}+\sum_{n=0}^{\infty}b_nz^n$, with $b_0 \neq 0$. Suppose that $\mu$ is an eigenvalue of the Sturm-Liouville problem and $y$ is a corresponding eigenvector. We proved in \cite{weir-2007} that $\mu \in \setR$. Now $y$ is a non-zero solution of \eqnref{eq:sl} in $(0,1)$ such that $\lim_{x\to 0+} y(x) = 0$ and $\lim_{x\to 1-} y(x)$ is finite. Since the space of solutions of \eqnref{eq:sl} is two-dimensional, $y=\alpha u + \beta u_1$ for some $\alpha$, $\beta \in \setC$. Considering the end-point $x=0$, we see that we must have $y=\alpha u$. Without loss of generality, we may assume $y = u$. Hence $u(x)$ converges to a finite limit as $x \to 1-$. Suppose that $\lambda$ is not an eigenvalue of $A_+$. Davies showed in \cite{davies-2007} that, for $\lambda \in \setR$, \eqnref{eq:recurrence} has two linearly independent solutions $\phi$, $\psi$ such that $\phi_n \geq n^{1/\varepsilon-1} \geq n^{-1}$ for all sufficiently large $n$ and $\mod{\psi_n} \sim n^{-1/\varepsilon-1}$ as $n \to \infty$. The space of solutions of \eqnref{eq:recurrence} is two-dimensional so $v_n = a \phi_n + b \psi_n$, and $a \neq 0$ since $\psi \in l^2(\setZ_+)$ and $v \notin l^2(\setZ_+)$. Without loss of generality $a>1$. Hence there exists $N > 0$ such that $v_n \geq n^{-1}$ for all $n \geq N$. For $x \in (0,1)$,
\begin{eqnarray*}
	u(x) & \geq & \sum_{n=1}^{N-1}(v_n - n^{-1}) x^n + \sum_{n=1}^{\infty}n^{-1}x^n\\
	     & = & \sum_{n=1}^{N-1}(v_n - n^{-1}) x^n - \log(1-x)\\
	     & \to & \infty
	\end{eqnarray*}
as $x \to 1-$. This is a contradiction, so $\lambda$ is an eigenvalue of $A_+$.
\qed
\end{proof}

\begin{corollary}\label{cor:correspondence}
$\lambda$ is an eigenvalue of $A_+$ if and only if $\mu = 2\lambda/\varepsilon$ is an eigenvalue of the Sturm-Liouville problem \eqnref{eq:sl}.
\end{corollary}

\section{Self-adjointness}

We now show that the operator corresponding to the Sturm-Liouville problem is essentially self-adjoint on a suitable domain. Equation \eqnref{eq:sl} can be written as
\begin{equation}
	Lu = \mu u
	\label{eq:op}
\end{equation}
where $L$ is an operator on $\Ltwow$ defined by
\begin{equation}
	Lf = -w^{-1}(pf')'
	\label{eq:ldef}
\end{equation}
on $\dom{L} = \mathcal{C}_0^{\infty}([0,1]) = \{f \in \smooth : f(0) = 0\} \subset \Ltwow$. We define
\[ \ip{f}{g}_w = \int_0^1 f(x) \conj{g(x)} w(x) \mathrm{d}x \]
and $\norm{f}_w = \ip{f}{f}_w^{1/2}$ for all $f$, $g \in \Ltwow$.

We also consider $L_c$, the restriction of $L$ to $\mathcal{C}_c^{\infty}(0,1)$, which is the space of smooth, compactly supported functions on $(0,1)$.

\begin{proposition}\label{prp:domains}
The adjoints $L_c^*$ of $L_c$ and $L^*$ of $L$ are closed extensions of $L$, which is symmetric, $L^* \subset L_c^*$, and the following are equivalent:

(a) $\mu \in \setC$ is an eigenvalue of the Sturm-Liouville operator $L$ and $u \in \dom{L}$ is a corresponding eigenvector;

(b) $\mu \in \setC$ is an eigenvalue of the Sturm-Liouville operator $\bar{L}$ and $u \in \dom{\bar{L}}$ is a corresponding eigenvector;

(c) $\mu \in \setC$ is an eigenvalue of the operator $L^*$ and $u \in \dom{L^*}$ is a corresponding eigenvector.

Moreover, if $0 < \varepsilon \leq 1$ then statements (a) -- (d) are equivalent to

(d) $\mu \in \setC$ is an eigenvalue of the operator $L_c^*$ and $u \in \dom{L_c^*}$ is a corresponding eigenvector.
\end{proposition}
\begin{proof}

We first show that $L$ is symmetric.  For all $f, g \in \dom{L}$ we have
\begin{eqnarray*}
	\ip{Lf}{g}_w 	& = & -\int_0^1 (pf')'(x) \overline{g(x)} \mathrm{d}x\\
								& = & -p(1)f'(1)\overline{g(1)} + p(0)f'(0)\overline{g(0)} + \int_0^1 f'(x)p(x)\overline{g'(x)} \mathrm{d}x\\
								& = & f(1)p(1)\overline{g'(1)} - f(0)p(0)\overline{g'(0)} -\int_0^1 f(x) \overline{(pg')'(x)} \mathrm{d}x\\
								& = & \ip{f}{Lg}_w
\end{eqnarray*}
since $p(1)=f(0)=g(0)=0$.

Now it is clear that $L_c$ is also symmetric and

\[ L_c \subseteq L \subseteq L^* \subseteq L_c^*, \]

the last two being closed.

We now prove the equivalence of statements (a)--(c):

(a) $\Rightarrow$ (b) $\Rightarrow$ (c): Immediate.

(c) $\Rightarrow$ (a): For all $\phi \in \mathcal{C}_c^{\infty}(0,1)$ we have
\begin{eqnarray*}
\int_0^1 u(x) (p\phi')'(x) \mathrm{d}x & = & -\ip{u}{L_c\conj{\phi}}_w = -\ip{(L_c)^*u}{\conj{\phi}}_w = -\mu \ip{u}{\conj{\phi}}_w\\
& = & -\mu\int_0^1 u(x)\phi(x)w(x)\mathrm{d}x
\end{eqnarray*}
so $(pu')' = -\mu wu$ when we consider $u$ as an element of the space $\mathcal{D}'(0,1)$ of distributions on the test-function space $\mathcal{C}_c^{\infty}(0,1)$. Since $u \in \Ltwow$, $w^{1/2}u \in \Ltwo \subset \Lone$. Also $w^{1/2} \in \mathcal{C}((0,1])$, so $wu \in L^1(\delta,1)$ for any $\delta \in (0,1)$. Therefore $pu' \in W^1(\delta,1)$ for any such $\delta$ and hence $pu'$ has a representation which is continuous on $(0,1]$ given by
\[(pu')(x) = -\int_x^1(pu')'(y) \mathrm{d}y + c = \mu \int_x^1 w(y)u(y) \mathrm{d}y + c\]
for some constant $c$ and all $x \in (0,1]$. Since $p$ is continuous on $[0,1]$ and $p>0$ on $[0,1)$, $u'$ is continuous on $(0,1)$, i.e. $u$ is continuously differentiable on $(0,1)$. It now follows from the above equation that $pu'$ is in fact continuously differentiable on $(0,1)$. Since $p$ is continuously differentiable and non-zero on $(0,1)$, we see that $u$ is twice differentiable in $(0,1)$ and hence a classical solution of equation \eqnref{eq:heun}. Considering the Frobenius expansions at the left-hand endpoint and the condition that $u \in \Ltwow$ we find that $u \in \mathcal{C}^{\infty}([0,1))$ with $u(0)=0$. Considering the Frobenius expansions at the right hand endpoint, we see that either $u(x) \sim 1$ as $x \to 1-$ or $u(x) \sim (1-x)^{-1/\varepsilon}$ as $x \to 1-$. We are required to show that it is the former which holds. We have
\begin{equation}
	\ip{Lf}{u}_w = \ip{f}{L^*u}_w
	\end{equation}
for all $f \in \dom{L}$. Since $u$ is smooth in $[0,1)$, $L^*u = \mu u = -w^{-1}(pu')'$ in the classical sense of differentiation. So
\begin{eqnarray*}
	-\int_0^1(pf')'(x)\conj{u(x)}\mathrm{d}x	& = &	-\int_0^1f(x)\conj{(pu')'(x)}\mathrm{d}x\\
																						& = &	-\left[f(x)p(x)\conj{u'(x)}\right]_0^{x \to 1-} + \int_0^1f'(x)p(x)\conj{g'(x)}\mathrm{d}x\\
																						& = & \left[f'(x)p(x)\conj{u(x)}-f(x)p(x)\conj{u'(x)}\right]_0^{x\to 1-} -\int_0^1(pf')'(x)\conj{g(x)}\mathrm{d}x
	\end{eqnarray*}
for all $f \in \dom{L}$. Since $f(0) = u(0) = 0$, this implies that
\begin{equation}
	\lim_{x \to 1-} p(x)\left[f'(x)\conj{u(x)} - f(x)\conj{u'(x)}\right] = 0
	\end{equation}
for all $f \in \dom{A}$. By choosing $f(x) = \sin(\pi x/2)$ we see that \[\lim_{x \to 1-} p(x)\conj{u'(x)} = 0.\] If $u(x) \sim (1-x)^{-1/\varepsilon}$ as $x \to 1-$ then $u'(x) \sim (1-x)^{-1-1/\varepsilon}$ as $x \to 1-$ and hence $p(x)\conj{u'(x)} \sim 1$ as $x \to 1-$. This is a contradiction, so $u(x) \sim 1$ as $x \to 1-$, as required.

We now assume $0 < \varepsilon \leq 1$. It is immediate that (c) implies (d). The proof that (d) implies (a) is similar to the proof that (c) implies (a), but the possibility that $u(x) \sim (1-x)^{-1/\varepsilon}$ as $x \to 1-$ is ruled out by the condition that $u \in \Ltwow$.
\qed
\end{proof}

\begin{theorem}\label{thm:sa}
The Sturm-Liouville operator $L$ is essentially self-adjoint. If $0 < \varepsilon \leq 1$ then $\bar{L}_c = \bar{L}$.
\end{theorem}
\begin{proof}
Suppose that $\mu$ is an eigenvalue of $L^*$. Then, by Proposition \ref{prp:domains}, $\mu$ is an eigenvalue of $L$ and hence real, since $L$ is symmetric. Hence the deficiency indices of $L$ are both zero, so $L$ is essentially self adjoint (see Theorem 1.2.7 in \cite{stdo}). If $0<\varepsilon\leq1$ then $L_c^*$ is also essentially self-adjoint, by the same argument. Since $\bar{L}$ is a self-adjoint extension of $L_c$, the result follows. \qed
\end{proof}

\begin{lemma}\label{lem:injective}
The operator $\bar{L}$ is injective.
\end{lemma}
\begin{proof}
Suppose for a contradiction that $\bar{L}$ is not injective. Then $0$ is an eigenvalue of $\bar{L}$. By Proposition \ref{prp:domains}, $0$ is also an eigenvalue of the classical Sturm-Liouville problem \eqnref{eq:sl} and hence of $A_+$ by the work in Section \ref{sec:correspondence}. Davies showed in \cite{davies-2007} that $\lambda>1$ for all real eigenvalues $\lambda$ of $A_+$, so this is a contradiction. \qed
\end{proof}

\section{Compactness of the resolvent}

In this section we give the integral kernel of the inverse of $\bar{L}$ explicitly, and use this to show that the resolvent is compact. This yields our result that the spectrum is discrete and the eigenvalues of $\bar{L}$ accumulate at $+\infty$.

We define $\gamma : [0,1] \to \setR \cup \{\infty\}$ by
\begin{equation}
	\gamma(x) = \int_0^x p(t)^{-1} \mathrm{d}t
\end{equation}
for all $x \in [0,1]$ and $G : [0,1] \times [0,1] \to \setR \cup \{\infty\}$ by
\begin{equation}
	G(x,y) = \left\{ \begin{array}{cc}	\gamma(x)	&	\textrm{if } x \leq y\\
																		\gamma(y)	&	\textrm{if } x \geq y
									\end{array}\right.
\end{equation}

\begin{lemma}\label{lem:props}
If $G$ is as above we have:

(i) $G(x,y) \in \setR$ for all $x,y \in [0,1]$ except when $x=y=1$;

(ii) $G(x,y) = G(y,x)$ for all $x,y \in [0,1]$;

(iii) $\frac{\partial}{\partial y} G(x,y) = \chi_{[0,x)}(y) p(y)^{-1}$ for all $x \in [0,1]$, $y \in (0,1) \setminus \{x\}$;

(iv) $\frac{\partial}{\partial x} G(x,y) = \chi_{[0,y)}(x) p(x)^{-1}$ for all $y \in [0,1]$, $x \in (0,1) \setminus \{y\}$.
\end{lemma}
\begin{proof}
(i) If $x < 1$ or $y < 1$ then $p^{-1}$ is bounded on $[0,\min\{x,y\}]$ and hence the integral is finite.

(ii) Immediate from the symmetry of the definition.

(iii) For $y \in (0,x)$, $\pdd{y}G(x,y) = \gamma'(y) = p(y)^{-1}$, whereas for $y \in (x,1)$, $\pdd{y}G(x,y) = \dd{y}\gamma(x) = 0$.

(iv) Similar to the proof of (iii). \qed

\end{proof}

\begin{theorem}\label{thm:inverse}
The operator $\bar{L}$ has a compact inverse $R$ given by
\begin{equation}
	(Rf)(x) = \int_0^1 G(x,y) f(y) w(y)\mathrm{d}y
	\label{eq:resolvent}
\end{equation}
for all $f \in \Ltwow$ and all $x \in [0,1)$.
\end{theorem}
\begin{proof}
We first prove that $G \in L^2([0,1]\times[0,1], w(x)\mathrm{d}x \times w(y)\mathrm{d}y)$, and hence that \eqnref{eq:resolvent} defines a Hilbert-Schmidt operator on $\Ltwow$. If $0 < y \leq x <1$ then
\begin{eqnarray*}
	\mod{G(x,y)}^2w(y)	& = & \left( \int_0^y \frac{\mathrm{d}t}{p(t)} \right)^2 w(y)\\
													& \leq & c_0 \left((1-y)^{-1/\varepsilon} - 1 \right)^2 w(y)\\
													& \leq & c_0 y^{-1} \left( 1 - (1-y)^{1/\varepsilon} \right) \left((1-y)^{-1/\varepsilon} - 1 \right) (y+1)^{-1/\varepsilon}
\end{eqnarray*}
for some constant $c_0$. Hence
\begin{small}
	\begin{eqnarray*}
		\int_y^1 \mod{G(x,y)}^2 w(x)\mathrm{d}xw(y)	& \leq & c_0 y^{-1} \left( 1 - (1-y)^{1/\varepsilon} \right) \left((1-y)^{-1/\varepsilon} - 1 \right) (y+1)^{-1/\varepsilon} \int_y^1 w(x) \mathrm{d}x\\
																								& \leq & c_1 y^{-2} \left(1-(1-y)^{1/\varepsilon} \right) \left((1-y)^{-1/\varepsilon} - 1 \right) (y+1)^{-1/\varepsilon} \int_y^1 (1-x)^{1/\varepsilon} \mathrm{d}x\\
																								& \leq & c_2 y^{-2} \left(1-(1-y)^{1/\varepsilon} \right) \left((1-y)^{-1/\varepsilon} - 1 \right) (y+1)^{-1/\varepsilon} (1-y)^{1+1/\varepsilon}\\
																								& \leq & c_2 y^{-2} (1-y) \left(1-(1-y)^{1/\varepsilon}\right)^2 (y+1)^{-1/\varepsilon}
	\end{eqnarray*}
\end{small}
for some constants $c_1$ and $c_2$ and all $y \in (0,1]$. As a function of $y$, this is continuous on $(0,1]$, and in a neighbourhood of $0$ we have
\[\int_y^1 \mod{G(x,y)}^2 w(x) \mathrm{d}x w(y) \leq c_3 y^{-\varepsilon/2} \left( \frac{1-(1-y)^{1/\varepsilon}}{y^{1-\varepsilon/4}} \right)^2 \leq c_4 y^{-\varepsilon/2} \]
for some constants $c_3$ and $c_4$. Since $\varepsilon < 2$ we conclude that
\[ \int_0^1 \int_y^1 \mod{G(x,y)}^2 w(x) \mathrm{d}x w(y) \mathrm{d}y < \infty \]
and hence
\[
	\int_0^1 \int_0^1 \mod{G(x,y)}^2 w(x)\mathrm{d}x w(y)\mathrm{d}y	= 2\int_0^1 \int_y^1 \mod{G(x,y)}^2 w(x)\mathrm{d}x w(y)\mathrm{d}y < \infty
\]
as required, since $G(y,x) =G(x,y)$ by Lemma \ref{lem:props}.

We now prove that $R$ is the inverse of $\bar{L}$. Suppose that $f \in \mathcal{C}_c^{\infty}(0,1)$. We have
\begin{equation}
	(Rf) (x) = \int_0^1 G(x,y)f(y)w(y)\mathrm{d}y
\end{equation}
and, since $f$ is zero in sufficiently small neighbourhoods of $0$ and $1$, it is easy to show that this is differentiable with
\begin{eqnarray*}
	(Rf)'(x)	& = & \int_0^1 \pdd{x} G(x,y)f(y)w(y)\mathrm{d}y\\
						& = & \int_0^1 \chi_{[0,y)}(x) p(x)^{-1} f(y) w(y)\mathrm{d}y\\
						& = & p(x)^{-1} \int_x^1 f(y) w(y)\mathrm{d}y
\end{eqnarray*}
by Lemma \ref{lem:props}. The last integral is smooth and vanishes in a neighbourhood of $1$, and $p(x)^{-1}$ is smooth on $[0,1)$, so this implies that $Rf \in \mathcal{C}^{\infty}([0,1])$. Also
\[(Rf)(0) = \int_0^1 \gamma(0) f(y)w(y)\mathrm{d}y = 0\]
since $\gamma(0) = 0$. Therefore $Rf \in \dom{L}$ and
\begin{eqnarray*}
	(LRf)(x)	& = & -w(x)^{-1}\dd{x} \left( p(x) \dd{x}\int_0^1 G(x,y) f(y) w(y)\mathrm{d}y \right) \\
						& = & -w(x)^{-1}\dd{x} \left( p(x) \int_0^1 \chi_{[0,y)}(x) p(x)^{-1} f(y) w(y)\mathrm{d}y \right)\\
						& = & -w(x)^{-1}\dd{x} \int_x^1 f(y) w(y)\mathrm{d}y\\
						& = & f(x)
\end{eqnarray*}
for all $f \in \mathcal{C}_c^{\infty}([0,1])$. If $f \in \Ltwow$, let $(f_n)$ be a sequence in $\mathcal{C}_c^{\infty}([0,1])$ such that $\norm{f_n - f}_w \to 0$ as $n \to 0$. Then $\norm{Rf_n - Rf}_w \to 0$ and $\norm{LRf_n - f}_w = \norm{f_n - f}_w \to 0$ as $n \to \infty$. Hence $Rf \in \dom{\bar{L}}$ and $\bar{L}Rf=f$.

Conversely, let $f \in \dom{\bar{L}}$. Then $R\bar{L}f \in \dom{\bar{L}}$ and $\bar{L}R\bar{L}f = \bar{L}f$. Now $R\bar{L}f = f$ since $\bar{L}$ is injective by Lemma \ref{lem:injective}.
\qed
\end{proof}

\begin{corollary}\label{cor:nonneg}
The Sturm-Liouville operator $\bar{L}$ is non-negative in the sense that $\spec{\bar{L}} \subseteq (0,\infty)$.
\end{corollary}
\begin{proof}
Since $\bar{L}$ has compact resolvent, it has empty essential spectrum, and since it is self-adjoint its spectrum is thus equal to the set of its eigenvalues. By Proposition \ref{prp:domains}, it is sufficient to show that all eigenvalues of $L$ are non-negative. If $\mu$ is an eigenvalue of $L$ and $f$ is a corresponding eigenvector with $\norm{f}_w = 1$ then
\begin{equation}
	\mu = \ip{Lf}{f}_w = -\int_0^1 (pf')'(x)\conj{f(x)}\mathrm{d}x = \int_0^1 p(x)\mod{f'(x)}^2 \mathrm{d}x > 0
	\end{equation}
since $p$ is non-negative on $[0,1]$. Note that the inequality is strict, since $f' = 0$ a.e. would imply that $f$ is constant and hence $0$, since $f(0) = 0$.
\qed
\end{proof}

\begin{corollary}\label{cor:evs}
There exists a complete orthonormal set of eigenvectors $\{f_n\}_{n=1}^{\infty}$ of $L$ with corresponding eigenvalues $\mu_n \geq 0$ which converge monotonically to $+\infty$ as $n \to \infty$.
\end{corollary}
\begin{proof}
The corresponding result for $\bar{L}$ is standard, and the result for $L$ follows by Proposition \ref{prp:domains}.
\qed
\end{proof}

\begin{corollary}\label{cor:hevs}
The operator $-iH$ defined in Section \ref{sec:intro} has infinitely many eigenvalues which can be enumerated $\{\lambda_n\}_{n=-\infty}^{\infty}$, in increasing order, such that $\lambda_0 = 0$, $\lambda_{-n} = -\lambda_n$ and $\lambda_n \to \infty$ as $n \to \infty$.
\end{corollary}
\begin{proof}
The eigenvalues of $-iH$ are the same as the eigenvalues of $A$, since the two operators are unitarily equivalent. If $\{\lambda_n\}_{n=1}^{\infty}$ are the eigenvalues of $A_+$ in increasing order then \eqnref{eq:a} tells us that the eigenvalues of $A$ are $\{\lambda_n\}_{n=-\infty}^{\infty}$, where $\lambda_0 = 0$ and $\lambda_{-n} = -\lambda_n$. It follows from Corollary \ref{cor:correspondence}, Proposition \ref{prp:domains} and Corollary \ref{cor:evs} that $\lambda_n = \varepsilon\mu_n/2 \to \infty$ as $n \to \infty$.
\qed
\end{proof}

\section{Eigenvalue asymptotics and numerics}

\subsection{Quadratic form formulation}

When considering self-adjoint operators, it is standard practice to obtain eigenvalue asymptotics and upper and lower eigenvalue bounds by quadratic form techniques and variational methods. Our first task is to identify the precise domain of the quadratic form associated with $\bar{L}$. We also show that $\mathcal{C}_c^{\infty}(0,1)$ is a form core for the associated form, which frequently allows us to restrict our attention to this simpler class of functions in the subsequent analysis.

We define a quadratic form $Q$ on $\Ltwow$ by
\begin{equation}
	Q(f) =	\int_0^1 p \mod{f'}^2 \mathrm{d}x
\end{equation}
for $f$ in the domain
\[ W^{1,2} = \{f \in \Ltwow :Q(f) < \infty\} \subset \{f \in \mathcal{C}(0,1): f' \in L^1_{\mathrm{loc}} \} \]
and define a norm
\[ \norm{u}_{W^{1,2}} = (Q(u) + \norm{u}_w^2)^{1/2}\]
on $W^{1,2}$.

\begin{theorem}
The domain $W^{1,2}$ is complete with respect to $\norm{\cdot}_{W^{1,2}}$.
\end{theorem}
\begin{proof}
Suppose that $(f_n)$ is a Cauchy sequence in $W^{1,2}$. Let $\delta \in (0,1/2)$. Then
\begin{eqnarray*}
\int_{\delta}^{1-\delta} \mod{f_n(x)}^2 \mathrm{d}x		& \leq &	\sup_{y \in (\delta,1-\delta)} w(y)^{-1} \int_{\delta}^{1-\delta} \mod{f_n(x)}^2 w(x) \mathrm{d}x,\\
\int_{\delta}^{1-\delta} \mod{f'_n(x)}^2 \mathrm{d}x		& \leq &	\sup_{y \in (\delta,1-\delta)} p(y)^{-1} \int_{\delta}^{1-\delta} \mod{f'_n(x)}^2 p(x) \mathrm{d}x
\end{eqnarray*}
so $f_n|_{(\delta,1-\delta)}$ lies in the Sobolev space $H^1(\delta,1-\delta)$. Similar inequalities show that $(f_n|_{(\delta,1-\delta)})$ is a Cauchy sequence in $H^1(\delta,1-\delta)$ and hence converges to some $h_{\delta} \in H^1(\delta,1-\delta)$. It is easy to show that if $\delta_1 < \delta_2$ then $h_{\delta_2} = h_{\delta_1}|_{(\delta_2,1-\delta_2)}$. Hence we may consistently define $h,k:(0,1) \to \setC$ by $h(x) = h_{1/m}(x)$ and $k(x) = h'_{1/m}(x)$ for any $m \in \setN$ such that $x \in (1/m,1-1/m)$. Clearly $h,k \in L^1_{\mathrm{loc}}(0,1)$. Let $\phi \in \mathcal{C}_c^{\infty}(0,1)$ and choose $m \in \setN$ large enough that $\supp{\phi} \subset (1/m,1-1/m)$. Then
\begin{eqnarray*}
\int_0^1 h(x) \phi'(x) \mathrm{d}x	& = &	\int_{1/m}^{1-1/m} h_{1/m}(x) \phi'(x) \mathrm{d}x\\
																		& = & -\int_{1/m}^{1-1/m} h_{1/m}'(x) \phi(x) \mathrm{d}x\\
																		& = & -\int_0^1 k(x) \phi(x) \mathrm{d}x.
\end{eqnarray*}
Since this holds for all $\phi \in \mathcal{C}_c^{\infty}(0,1)$, $k = h'$.

Since $(f_n)$ is Cauchy in $W^{1,2}$, $Q(f_n)$ and $\norm{f_n}_w^2$ are bounded. Let $C_1 = \sup_n Q(f_n)$ and $C_2 = \sup_n \norm{f_n}_w^2$. For all $m \in \setN$,
\[ \int_{1/m}^{1-1/m} \mod{h(x)}^2 w(x) \mathrm{d}x \leq \sup_n \int_{1/m}^{1-1/m} \mod{f_n(x)}^2 w(x) \mathrm{d}x \leq \sup_n \norm{f_n}_w^2 = C_2\]
and hence $\norm{h}_w^2 \leq C_2$ by the Monotone Convergence Theorem. Also for all $m \in \setN$,
\[ \int_{1/m}^{1-1/m} \mod{h'(x)}^2 p(x) \mathrm{d}x \leq \sup_n \int_{1/m}^{1-1/m} \mod{f'_n(x)}^2 p(x) \mathrm{d}x \leq \sup_n Q(f_n) = C_1 \]
and hence $Q(h) \leq C_1$ by the Monotone Convergence Theorem. Therefore $h \in W^{1,2}$.

Let $\eta > 0$ be given. Since $(f_n)$ is a Cauchy sequence in $W^{1,2}$, there exists $N \in \setN$ such that
\[\int_0^1 \mod{f_{n_1}-f_{n_2}}^2 w(x) \mathrm{d}x < \eta/6\]
and
\[\int_0^1 \mod{f'_{n_1}-f'_{n_2}}^2 p(x) \mathrm{d}x < \eta/6\]
for all $n_1,n_2 \geq N$. By the Dominated Convergence Theorem,
\[ \int_0^{1/m} \mod{f_N - h}^2 w(x) \mathrm{d}x + \int_{1-1/m}^1 \mod{f_N - h}^2 w(x) \mathrm{d}x \to 0\]
as $m \to \infty$. Hence there exists $M \in \setN$ such that
\[ \int_0^{1/m} \mod{f_N - h}^2 w(x) \mathrm{d}x + \int_{1-1/m}^1 \mod{f_N - h}^2 w(x) \mathrm{d}x < \eta/6 \]
for all $m \geq M$. If $n \geq N$ then
\begin{eqnarray*}
\int_{(0,1/m) \cup (1-1/m,1)} \mod{f_n - h}^2 w(x) \mathrm{d}x	& \leq & 2\int_{(0,1/m) \cup (1-1/m,1)} \mod{f_n - f_N}^2 w(x) \mathrm{d}x\\
																																&   & + 2\int_{(0,1/m) \cup (1-1/m,1)} \mod{f_N - h}^2 w(x) \mathrm{d}x\\
																																& \leq & 2 \int_0^1 \mod{f_n - f_N}^2 w(x) \mathrm{d}x\\
																																&   & + 2\int_{(0,1/m) \cup (1-1/m,1)} \mod{f_N - h}^2 w(x) \mathrm{d}x\\
																																& < & 2\eta/6 + 2\eta/6\\
																																& = & 2\eta/3
\end{eqnarray*}
for all $m \geq M$. Since $w$ is bounded on $(1/M,1-1/M)$,
\begin{eqnarray*}
\int_{1/M}^{1-1/M} \mod{f_n - h}^2 w(x) \mathrm{d}x	& \leq &	\sup_{y \in (1/M,1-1/M)} w(y) \int_{1/M}^{1-1/M} \mod{f_n - h}^2 \mathrm{d}x \to 0\\
																										& = &		\sup_{y \in (1/M,1-1/M)} w(y) \int_{1/M}^{1-1/M} \mod{f_n - h_{1/M}}^2 \mathrm{d}x \to 0
\end{eqnarray*}
as $n \to \infty$. Hence there exists $N'$ such that
\[\int_{1/M}^{1-1/M}\mod{f_n-h}^2 w(x) \mathrm{d}x < \eta/3\]
for all $n \geq N'$. Therefore
\[\int_0^1 \mod{f_n-h}^2 w(x) \mathrm{d}x < \eta\]
for all $n \geq \max(N,N')$, so $\norm{f_n-h}_w^2 \to 0$ as $n \to \infty$. Similarly $Q(f_n-h) \to 0$ as $n \to \infty$. Hence $f_n \stackrel{W^{1,2}}{\longrightarrow} h \in W^{1,2}$ as $n \to \infty$. \qed
\end{proof}

\begin{theorem}\label{thm:cts}
If $f \in W^{1,2}$ then $f$ is continuous on $[0,1)$ and $f(0)=0$.
\end{theorem}
\begin{proof}
For any $\alpha \in (0,1)$, $p^{-1/2} \in L^2(0,\alpha)$ and $p^{1/2}f' \in L^2(0,\alpha)$. Hence $f' \in L^1(0,\alpha)$, so $f$ is continuous on $[0,\alpha]$. Since $\alpha$ is arbitrary in $(0,1)$, $f$ is continuous on $[0,1)$.

Suppose that $f(0) \neq 0$. Then $\mod{f}^2 \geq c$ on $[0,\delta]$ for some $\delta > 0$ and some $c > 0$. Then
\[\int_0^{\delta} w(x) \mod{f(x)}^2 \mathrm{d}x \geq c \int_0^{\delta} w(x) \mathrm{d}x = \infty.\]
However,
\[\int_0^{\delta} w(x) \mod{f(x)}^2 \mathrm{d}x = \norm{f}_w^2 < \infty.\]
This is a contradiction, so $f(0) = 0$. \qed
\end{proof}

\begin{theorem}\label{thm:core}
The space $\mathcal{C}_c^{\infty}(0,1)$ is dense in $W^{1,2}$.
\end{theorem}
\begin{proof}
For all $\delta$ such that $0 < \delta < 1/3$, define $f_{\delta}$ by
\[ f_{\delta}(x) =	\left\{ 	\begin{array}{cc}
																0								&	0 \leq x \leq \delta\\
																f(2(x-\delta))	&	\delta \leq x \leq 2\delta\\
																f(x)						&	2\delta \leq x \leq 1-\delta\\
																f(1-\delta)			&	1-\delta \leq x \leq 1.
															\end{array}
										\right. \]
Then $f_{\delta}|_{(\delta,1)} \in H^1(\delta,1)$ and $f_{\delta}(\delta)=0$. It follows from the boundedness of $p$ and $w$ on $[\delta,1]$ that $f_{\delta} \in W^{1,2}$. We have
\begin{eqnarray}
\int_0^1 \mod{f(x)-f_{\delta}(x)}^2 w(x) \mathrm{d}x	& = & \int_0^{\delta} \mod{f(x)}^2 w(x) \mathrm{d}x\\
																											&   & + \int_{\delta}^{2\delta} \mod{f(x)-f(2(x-\delta))}^2 w(x) \mathrm{d}x\\
																											&   &	+ \int_{1-\delta}^1 \mod{f(x)-f(1-\delta)}^2 w(x) \mathrm{d}x \label{eq:endint}
\end{eqnarray}
for all $\delta \in (0,1/3)$. That the first integral on the right hand side converges to zero as $\delta \to 0$ is elementary. On some neighbourhood of $0$, $w$ is monotone decreasing. Thus, if we take $\delta$ sufficiently small, then $w(x/2+\delta) \leq w(x)$ for all $x \in (0,2\delta)$. Hence, for all small enough $\delta$,
\begin{eqnarray*}
\int_{\delta}^{2\delta} \mod{f(x)-f(2(x-\delta))}^2 w(x) \mathrm{d}x	& \leq &	2 \int_{\delta}^{2\delta} \mod{f(2(x-\delta))}^2 w(x) \mathrm{d}x + 2 \int_{\delta}^{2\delta} \mod{f(x)}^2 w(x) \mathrm{d}x\\
																																			& =	&			\int_0^{2\delta} \mod{f(x)}^2 w(x/2+\delta) \mathrm{d}x + 2 \int_{\delta}^{2\delta} \mod{f(x)}^2 w(x) \mathrm{d}x\\
																																			& \leq &	\int_0^{2\delta} \mod{f(x)}^2 w(x) \mathrm{d}x + 2 \int_{\delta}^{2\delta} \mod{fx)}^2 w(x) \mathrm{d}x\\
																																			& \to &		0
\end{eqnarray*}
as $\delta \to 0$. In order to with integral \eqnref{eq:endint}, we must first obtain an estimate on $f$ near $1$. By Theorem \ref{thm:cts}, $f(0) = 0$. Hence
\begin{eqnarray*}
\mod{f(x)}	& \leq &	\int_0^x \mod{f'(s)} \mathrm{d}s\\
						& \leq &	\left(\int_0^x p(s)^{-1} \mathrm{d}s\right)^{1/2}\left(\int_0^x p(s)\mod{f'(s)}^2\mathrm{d}s \right)^{1/2}
\end{eqnarray*}
for all $x \in (0,1)$, so
\begin{eqnarray*}
\mod{f(x)}^2	& \leq &	Q(f) \int_0^x p(s)^{-1} \mathrm{d}s\\
							& \leq &	c Q(f) (1-x)^{-1/\varepsilon}
\end{eqnarray*}
for some constant $c$. Now, since $w$ is monotone decreasing in a neighbourhood of $1$ and there exists a constant $c'$ such that $w(s) \leq c' (1-s)^{1/\varepsilon}$ for all $s > 1/2$,
\begin{eqnarray*}
\int_{1-\delta}^1 \mod{f(x) - f(1-\delta)}^2 w(x) \mathrm{d}x	& \leq &	2 \int_{1-\delta}^1 \mod{f(x)}^2 w(x) \mathrm{d}x\\
																															&      &	+ 2\int_{1-\delta}^1 \mod{f(1-\delta)}^2 w(x) \mathrm{d}x\\
																															& \leq &	2 \int_{1-\delta}^1 \mod{f(x)}^2 w(x) \mathrm{d}x\\
																															&      &	+ 2\int_{1-\delta}^1 \mod{f(1-\delta)}^2 w(1-\delta) \mathrm{d}x\\
																															& \leq &	2 \int_{1-\delta}^1 \mod{f(x)}^2 w(x) \mathrm{d}x\\
																															&      &	+ 2cc' Q(f) \delta (1-\delta)^{-1/\varepsilon} (1-\delta)^{1/\varepsilon}\\
																															& \to  &	0
\end{eqnarray*}
as $\delta \to 0$. Therefore $\norm{f - f_{\delta}} \to 0$ as $\delta \to 0$. Similarly, we may show that $Q(f-f_{\delta}) \to 0$ as $\delta \to 0$. Hence $f_{\delta} \stackrel{W^{1,2}}{\longrightarrow} f$ as $\delta \to 0$.

Next we define $f_{\delta,\eta}$ by
\[ f_{\delta,\eta}(x) = 	\left\{	\begin{array}{cc}
																		f_{\delta}(x)																								&	0 \leq x \leq 1-2\eta\\
																		f_{\delta}(1-2\eta)\left(1-\frac{x-(1-2\eta)}{\eta}\right)	& 1-2\eta \leq x \leq 1-\eta\\
																		0																														& 1-\eta \leq x \leq 1
																	\end{array}
													\right. \]
for all $\delta, \eta < 1/3$. Then $f_{\delta,\eta}|_{(\delta,1-\eta)} \in H^1(\delta,1-\eta)$, $f(x) = 0$ on $(0,\delta) \cup (1-\eta,1)$, and $\lim_{x \to \delta+} f_{\delta,\eta}(x) = \lim_{x \to 1-\eta-} f_{\delta,\eta}(x) = 0$. Hence $f_{\delta,\eta} \in H^1(0,1)$. By Theorem 2 of Section 5.5 of \cite{evans}, there exists a sequence $(f_{\delta,\eta,n})_{n \in \setN}$ in $\mathcal{C}_c^{\infty}(\delta,1-\eta) \subset \mathcal{C}_c^{\infty}(0,1)$ such that $f_{\delta,\eta,n} \stackrel{H^1}{\longrightarrow} f_{\delta,\eta}$ as $n \to \infty$. Since $p$ and $w$ are bounded on $[\delta,1-\eta]$, it follows that $f_{\delta,\eta,n} \stackrel{W^{1,2}}{\longrightarrow} f_{\delta,\eta}$ as $n \to \infty$.

It only remains to show that we may approximate $f_{\delta}$ by $f_{\delta,\eta}$. For all $\delta \in (0,1/3)$,
\begin{eqnarray}
\int_0^1 \mod{f_{\delta}(x)-f_{\delta,\eta}(x)}^2 w(x)\mathrm{d}x	& = &	\int_{1-2\eta}^{1-\eta} \mod{f_{\delta}(x)-f_{\delta,\eta}(x)}^2 w(x)\mathrm{d}x\\
																																	&   & + \int_{1-\eta}^1 \mod{f_{\delta}(x)}^2 w(x) \mathrm{d}x.\label{eq:rightend}
\end{eqnarray}
Clearly integral \eqnref{eq:rightend} converges to zero as $\eta \to 0$ and
\begin{eqnarray*}
\int_{1-2\eta}^{1-\eta} \mod{f_{\delta}(x)-f_{\delta,\eta}(x)}^2 w(x)\mathrm{d}x	& \leq &	2\int_{1-2\eta}^{1-\eta} \mod{f_{\delta}(x)}^2 w(x)\mathrm{d}x\\
																																									&      &	+2\int_{1-2\eta}^{1-\eta} \mod{f_{\delta,\eta}(x)}^2 w(x)\mathrm{d}x\\
																																									& \leq &	2\int_{1-2\eta}^{1-\eta} \mod{f_{\delta}(x)}^2 w(x)\mathrm{d}x\\
																																									&      &	+2\int_{1-2\eta}^{1-\eta} \mod{f_{\delta}(1-2\eta)}^2 w(x)\mathrm{d}x\\
																																									& \leq &	2\int_{1-2\eta}^{1-\eta} \mod{f_{\delta}(x)}^2 w(x)\mathrm{d}x\\
																																									&      &	+2\norm{f_{\delta}}_{L^{\infty}} \int_{1-2\eta}^{1-\eta} w(x) \mathrm{d}x\\
																																									& \to	 & 0
\end{eqnarray*}
as $\eta \to 0$. Hence $\norm{f_{\delta}-f_{\delta,\eta}} \to 0$ as $\eta \to 0$. If $0 < \eta < \delta$ then
\begin{eqnarray*}
\int_0^1 \mod{f'_{\delta}(x) - f'_{\delta,\eta}(x)}^2 p(x) \mathrm{d}x	& = &			\int_{1-2\eta}^{1-\eta} \mod{f'_{\delta}(x) - f'_{\delta,\eta}(x)}^2 p(x) \mathrm{d}x\\
																																				&   &			+\int_{1-\eta}^1 \mod{f'_{\delta}(x)}^2 p(x) \mathrm{d}x\\
																																				& \leq &	2\int_{1-2\eta}^{1-\eta} \mod{f'_{\delta}(x)}^2 p(x) \mathrm{d}x\\
																																				&   &			+2\int_{1-2\eta}^{1-\eta} \mod{f'_{\delta,\eta}(x)}^2 p(x) \mathrm{d}x\\
																																				& \leq &	2\int_{1-2\eta}^{1-\eta} \mod{f'_{\delta}(x)}^2 p(x) \mathrm{d}x\\
																																				&   &			+2\eta^{-2}\mod{f_{\delta}(1-2\eta)}^2 \int_{1-2\eta}^{1-\eta} p(x) \mathrm{d}x\\
																																				& \leq &	2\int_{1-2\eta}^{1-\eta} \mod{f'_{\delta}(x)}^2 p(x) \mathrm{d}x\\
																																				&   &			+2\eta^{-1}\norm{f_{\delta}}_{L^{\infty}}^2 \sup_{1-2\eta\leq x \leq1} p(x)\\
																																				& \leq &	2\int_{1-2\eta}^{1-\eta} \mod{f'_{\delta}(x)}^2 p(x) \mathrm{d}x\\
																																				&   &			+2\eta^{-1}\norm{f_{\delta}}_{L^{\infty}}^2 (2\eta)^{1+1/\varepsilon}2^{1-1/\varepsilon}\\
																																				& = &			2\int_{1-2\eta}^{1-\eta} \mod{f'_{\delta}(x)}^2 p(x) \mathrm{d}x\\
																																				&   &			+8\eta^{1/\varepsilon}\norm{f_{\delta}}_{L^{\infty}}^2\\
																																				& \to &		0
\end{eqnarray*}
as $\eta \to 0$. Therefore $f_{\delta,\eta} \stackrel{W^{1,2}}{\longrightarrow} f_{\delta}$ as $\eta \to 0$. \qed
\end{proof}

\begin{corollary}
The quadratic form associated with $\bar{L}$ is $Q$.
\end{corollary}
\begin{proof}
Let $\tilde{Q'}$ be the form defined on $\dom{L}$ by $\tilde{Q'}(f,g) = \ip{Lf}{g}_w$. Then by Theorem 4.4.5 of \cite{stdo}, $\tilde{Q'}$ is closable and its closure is associated with a self-adjoint extension of $L$, which must be $\bar{L}$, since $L$ is essentially-self adjoint. For $f \in \dom{L}$,
\begin{eqnarray*}
\tilde{Q}(f) 	& = & \ip{Lf}{f}_w\\
							& = & -\int_0^1 (pf')'(x) \conj{f(x)} \mathrm{d}x\\
							& = & \int_0^1 p \mod{f'}^2 \mathrm{d}x\\
							& = & Q(f)
\end{eqnarray*}
so $Q$ is an extension of $\tilde{Q}$. By the above theorem,
\[ \mathcal{C}_c^{\infty}(0,1) \subseteq \dom{\tilde{Q}} \subseteq \dom{Q} = \overline{\mathcal{C}_c^{\infty}(0,1)} \]
and hence $Q$ is the closure of $\tilde{Q}$, as required.
\qed
\end{proof}

\subsection{Transformation to a Schr\"odinger Operator}

We next use a suitable change of variables to convert $\bar{L}$ into a Schr\"odinger operator on a certain space $L^2(0,\beta)$. This allows us to use the extensive range of standard techniques available for controlling the eigenvalues of Schr\"odinger operators.

\begin{theorem}\label{thm:schr}
Define
\[\psi(t) = \int_0^t{ \sqrt{\frac{w(y)}{p(y)}}\mathrm{d}y}\]
for $t \in [0,1]$, and let $\beta = \psi(1)$. Then $\beta < \infty$ and $\psi:[0,1]\to[0,\beta]$ is invertible. Let $\phi = \psi^{-1}$ and define
\begin{eqnarray*}
c(s) & = & \left(w(\phi(s))p(\phi(s))\right)^{-1/4}\\
V(s) & = & -c(s)c''(s)\frac{p(\phi(s))}{\phi'(s)} - c(s)c'(s)\dd{s}\frac{p(\phi(s))}{\phi'(s)}\\
\hat{Q}_c(g) & = & \int_0^{\beta} \left\{\mod{g'(s)}^2 + V(s)\mod{g(s)}^2 \right\}\mathrm{d}s
\end{eqnarray*}
for all $g \in \mathcal{C}_c^{\infty}(0,\beta) \subset L^2(0,\beta)$. Then $\hat{Q}_c$ is a closable quadratic form and its closure $\hat{Q}$ is associated with a self-adjoint operator $H$, which is unitarily equivalent to $\bar{L}$. The potential $V$ is smooth on $(0,1)$, $V(s) \sim \frac{3}{4}s^{-2}$ as $s \to 0+$ and $V(s) \sim \left(\frac{1}{\varepsilon^2}-\frac{1}{4}\right)(\beta-s)^{-2}$ as $s \to \beta-$.
\end{theorem}
\begin{proof}
We have
\begin{eqnarray*}
\beta	& = 		& \int_0^1 y^{-1/2}(1-y)^{-1/2}(1+y)^{-1/2} \mathrm{d}y\\
			& \leq 	& \sqrt{2}\int_0^{1/2} y^{-1/2} \mathrm{d}y + \sqrt{2}\int_{1/2}^1(1-y)^{-1/2} \mathrm{d}y\\
			& = 		& 2\sqrt{2}\int_0^{1/2} y^{-1/2} \mathrm{d}y = 4.
\end{eqnarray*}
Since $w(t)^{1/2}p(t)^{-1/2}$ is smooth and positive for all $t \in (0,1)$, it is immediate from its definition that $\psi$ is smooth on $(0,1)$ and continuous and strictly monotone increasing on $[0,1]$. Thus $\psi$ is injective. Since $\psi$ is continuous, $\psi(0) = 0$ and $\psi(1) = \beta$, $\psi$ must be surjective. Hence $\psi$ is invertible. It is easy to show that $\phi$ and $\phi'$ are smooth and non-zero.

We define
\[(Uf)(s) = c(s)^{-1}f(\phi(s))\]
for all $f \in \Ltwow$ and all $s \in [0,\beta]$.
It follows from the definitions of $\psi$ and $\phi$ that $\psi'(t) = w(t)^{1/2}p(t)^{-1/2}$ and $\phi'(s) = p(\phi(s))^{1/2}w(\phi(s))^{-1/2}$. Hence, making the change of variables $x = \phi(s)$,
\begin{eqnarray*}
\norm{f}_w^2	& = & \int_0^1 \mod{f(x)}^2 w(x) \mathrm{d}x\\
							& = & \int_{\psi(0)}^{\psi(1)} \mod{f(\phi(s))}^2 w(\phi(s)) \phi'(s) \mathrm{d}s\\
							& = & \int_0^{\beta} \mod{f(\phi(s))}^2 (w(\phi(s))p(\phi(s)))^{1/2} \mathrm{d}s\\
							& = & \int_0^{\beta} \mod{c(s)^{-1}f(\phi(s))}^2 \mathrm{d}s\\
							& = & \norm{Uf}_{L^2}^2
\end{eqnarray*}
for all $f \in \Ltwow$. It follows that $U$ is a unitary operator from $\Ltwow$ to $L^2(0,\beta)$.

Since $p$ and $w$ are smooth and non-zero on $(0,1)$ and $\phi$ is smooth and non-zero on $(0,\beta)$, $c$ is smooth and non-zero on $(0,\beta)$. Hence $Uf \in \mathcal{C}_c^{\infty}(0,\beta)$ if and only if $f \in \mathcal{C}_c^{\infty}(0,1)$. Making the change of variables $x = \phi(s)$ as above,
\begin{eqnarray*}
Q(f)	& = & \int_0^1 \mod{f'(x)}^2 p(x) \mathrm{d}x\\
			& = & \int_0^{\beta} \mod{c'(s)(Uf)(s)+c(s)(Uf)'(s)}^2 \psi'(\phi(s))^2 p(\phi(s)) \phi'(s) \mathrm{d}s\\
			& = & \int_0^{\beta} \mod{c'(s)(Uf)(s)+c(s)(Uf)'(s)}^2 \phi'(s)^{-1} p(\phi(s)) \mathrm{d}s\\
			& = & \int_0^{\beta} \left\{ c(s)^2\mod{(Uf)'(s)}^2 + c'(s)c(s)\dd{s}\left(\mod{(Uf)(s)}^2\right) + c'(s)^2\mod{(Uf)(s)}^2 \right\} \frac{p(\phi(s))}{\phi'(s)} \mathrm{d}s\\
			& = & \int_0^{\beta} \left\{ c(s)^2\mod{(Uf)'(s)}^2 + (c'(s)^2-(cc')'(s))\mod{(Uf)(s)}^2 \right\} \frac{p(\phi(s))}{\phi'(s)}\\
			&   & -c(s)c'(s)\dd{s}\frac{p(\phi(s)}{\phi'(s)}\mod{(Uf)(s)}^2 \mathrm{d}s\\
			& = & \int_0^{\beta} \left\{\mod{(Uf)'(s)}^2 + V(s)\mod{(Uf)(s)}^2 \right\}\mathrm{d}s\\
			& = & \hat{Q}_c(Uf)
\end{eqnarray*}
for all $f \in \mathcal{C}_c^{\infty}(0,1)$.

Since $\hat{Q}_c$ is the form arising from the symmetric operator
\[ H_c = -\Delta + V \]
with domain $\dom{H_c} = \mathcal{C}_c^{\infty}(0,\beta)$ and
\[ \ip{H_cf}{f} = \hat{Q}_c(f) = Q(U^{-1}f) \geq 0 \]
for all $f \in \mathcal{C}_c^{\infty}(0,\beta)$, $\hat{Q}_c$ is closable and its closure, $\hat{Q}$, is associated with a non-negative self-adjoint extension $H$ of $H_c$ by Theorem 4.4.5 of \cite{stdo}.

We have proven that $f \in \mathcal{C}_c^{\infty}(0,1)$ if and only if $Uf \in \mathcal{C}_c^{\infty}(0,\beta)$, and
\begin{equation}\label{eq:equiv}
Q(f) = \hat{Q}_c(Uf)
\end{equation}
for all $f \in \mathcal{C}_c^{\infty}(0,1)$.
We now prove that $Uf \in \dom{\hat{Q}}$ if and only if $f \in \dom{Q}$ and that
\begin{equation}
Q(f) = \hat{Q}(Uf)
\end{equation}
for all $f \in \dom{Q}$.
By Theorem \ref{thm:core}, $\mathcal{C}_c^{\infty}(0,1)$ is dense in $\dom{Q}$ with respect to $\norm{\cdot}_{W^{1,2}}$. Let $f \in \dom{Q}$. Then there is a sequence $(f_n)$ in $\mathcal{C}_c^{\infty}(0,1)$ such that $f_n \stackrel{W^{1,2}}{\longrightarrow} f$ as $n \to \infty$. Hence $(f_n)$ is a Cauchy sequence in $\mathcal{C}_c^{\infty}(0,1)$ with respect to $\norm{\cdot}_{W^{1,2}}$. By \eqnref{eq:equiv} and the unitarity of $U$, $(Uf_n)$ is a Cauchy sequence in $\mathcal{C}_c^{\infty}(0,\beta)$ with respect to $\tnorm{\cdot} = (\norm{\cdot}^2_{L^2} + \hat{Q}(\cdot))^{1/2}$. Since $\dom{\hat{Q}}$ is the completion of $\mathcal{C}_c^{\infty}(0,\beta)$ with respect to $\tnorm{\cdot}$, there exists $g \in \dom{\hat{Q}}$ such that $\tnorm{Uf_n - g} \to 0$ as $n \to \infty$. Also $\norm{f_n - f}_w \to 0$ as $n \to \infty$ implies that $\norm{Uf_n - Uf}_{L^2} \to 0$ as $n \to \infty$. It follows that $Uf = g \in \dom{\hat{Q}}$. The converse is similar. For all $f \in \dom{Q}$,
\begin{eqnarray*}
\hat{Q}(Uf) & = & \tnorm{Uf}^2 - \norm{Uf}^2_{L^2}\\
						& = & \lim_{n \to \infty} \tnorm{Uf_n}^2 - \lim_{n \to \infty} \norm{Uf_n}^2_{L^2}\\
						& = & \lim_{n \to \infty} \norm{f_n}^2_{W^{1,2}} - \lim_{n \to \infty} \norm{f_n}^2_w\\
						& = & \norm{f}^2_{W^{1,2}}-\norm{f}^2_w\\
						& = & Q(f)
\end{eqnarray*}
where, as before, $(f_n)$ is a sequence in $\mathcal{C}_c^{\infty}(0,1)$ such that $f_n \stackrel{W^{1,2}}{\longrightarrow} f$.

It now follows from the polarisation identity for sesquilinear forms that
\[ \hat{Q}'(Uf,Ug) = Q'(f,g) \]
for all $f, g \in \dom{Q}$. It follows immediately that $Uf \in \dom{H}$ if and only if $f \in \dom{\bar{L}}$, and that $H = U\bar{L}U^{-1}$.
Since $c$, $c'$, $c''$, $p$, $\phi$ and $1/\phi'$ are smooth on $(0,1)$, it follows that $V$ is smooth on $(0,1)$. For the asymptotics of $V$, see Appendix \ref{app:asymptotics}.
\qed
\end{proof}

\subsection{Eigenvalue asymptotics}

Throughout this section, $(\mu_n)_{n=1}^{\infty}$ shall be the eigenvalues of $L$, or equivalently of $\bar{L}$ by Proposition \ref{prp:domains}, listed in increasing order and repeated according to multiplicity as in Corollary \ref{cor:evs}. By Theorem \ref{thm:schr}, $\bar{L}$ is unitarily equivalent to a Schr\"odinger operator $H$ with Dirichlet boundary conditions. We shall use the Rayleigh-Ritz variational formula to obtain bounds on $\mu_n$ in terms of the Dirichlet eigenvalues of $-\Delta$ on various intervals. Recall that the Dirichlet eigenvalues of $-\Delta$ on the interval $[a,b]$ are $\{n^2\pi^2(b-a)^{-2} : n \in \setN \}$ with corresponding eigenfunctions $\{\sin(n\pi(x-a)/(b-a)) : n \in \setN\}$. We quote the variational formula from \cite{stdo}:

If $K$ is a non-negative self-adjoint operator on a Hilbert space $\mathcal{H}$ and $M$ is a finite-dimensional subspace of $\dom{K}$ then we define
\[\lambda(M) = \sup \{\ip{Kf}{f} : f \in M \textrm{ and } \norm{f} = 1 \} \]
and
\[\lambda_n = \inf \{\lambda(M): M \subseteq \dom{K} \textrm{ and dim}(M) = n \}. \]
If $\lambda_n \to +\infty$ as $n \to \infty$ then $K$ has compact resolvent and the numbers $\lambda_n$ coincide with the eigenvalues of $K$ written in increasing order and repeated according to multiplicity. Since $K$ is non-negative and self-adjoint, it is associated with a closed quadratic form $Q$. If $\mathcal{D}$ is a core for $Q$, that is, a subspace of the domain of $Q$ such that the closure of $Q$ restricted to $\mathcal{D}$ is $Q$, then we have
\[ \lambda_n = \inf \{\lambda(M): M \subseteq \mathcal{D} \textrm{ and dim}(M) = n \} \]
for all $n \in \setN$.

\begin{theorem}\label{thm:muasym}
Let $\beta = \int_0^1 y^{-1/2}(1-y)^{-1/2}(1+y)^{-1/2} \mathrm{d}y$, as in Theorem \ref{thm:schr}. Then
\[\lim_{n \to \infty}n^{-2}\mu_n = \frac{\pi^2}{\beta^2}.\]
Indeed, if $\alpha = \min\{V(s):s \in (0,\beta)\}$, where $V$ is as in Theorem \ref{thm:schr}, then
\[ \mu_n \geq \frac{n^2\pi^2}{\beta^2} + \alpha\]
for all $n \in \setN$, and
\[ \mu_n \leq \frac{n^2\pi^2}{\beta^2} + O(n^{4/3}) \]
as $n \to \infty$.
\end{theorem}
\begin{proof}
Define
\[Q_{\alpha}(g) = \ip{Hg}{g} - \alpha\ip{g}{g} = \hat{Q}_c(g) - \alpha\ip{g}{g} \geq 0\]
for all $g \in \mathcal{C}_c^{\infty}(0,\beta)$. Then $Q_{\alpha}$ is a closable form and its closure $\conj{Q}_{\alpha}$ is the form associated with a non-negative self-adjoint extension $H_{\alpha}$ of $H_c -\alpha I$, where $H_c = -\Delta + V$ with $\dom{H_c} = \mathcal{C}_c^{\infty}(0,\beta)$. Since $\mathcal{C}_c^{\infty}(0,\beta)$ is a core for both $\hat{Q}$ and $\bar{Q}_{\alpha}$, the variational formula implies that the $n$th eigenvalue of $H_{\alpha}$ is $\mu_n - \alpha$.

We now define
\[\tilde{Q}(g) = \int_0^{\beta} \mod{g'(s)}^2 \mathrm{d}s \leq \int_0^{\beta} \mod{g'(s)}^2 +(V(s)-\alpha)\mod{g(s)}^2 \mathrm{d}s = Q_{\alpha}(g) \]
for all $g \in \mathcal{C}_c^{\infty}(0,\beta)$. The closure $\conj{\tilde{Q}}$ of $\tilde{Q}$ is the form associated with the operator $-\Delta$ on $L^2(0,\beta)$ with Dirichlet boundary conditions. Since $\mathcal{C}_c^{\infty}(0,\beta)$ is a core for both $\conj{\tilde{Q}}$ and $\bar{Q}_{\alpha}$, the variational formula implies that
\[\frac{n^2\pi^2}{\beta^2} \leq \mu_n - \alpha \]
for all $n \in \setN$.

For all sufficiently small $\delta > 0$ define
\[ K_{\delta}(f) = (-\Delta + c_{\delta}I)(f) \]
for all $f \in \mathcal{C}_c^{\infty}(\delta,\beta-\delta)$, where
\[ c_{\delta} = \sup\{V(s):s \in (\delta,\beta-\delta).\]
Then $K_{\delta}$ is a non-negative symmetric operator and hence has a non-negative self-ajoint extension $\tilde{K}_\delta$, called the Friedrichs extension of $K_{\delta}$, associated with the closure of the form $Q_{\delta}$ defined by
\[ Q_{\delta}(f) = \ip{K_{\delta}f}{f} \]
for all $f \in \mathcal{C}_c^{\infty}(\delta,\beta-\delta)$.
For all finite dimensional subspaces $M$ of $\dom{K_{\delta}}$, define
\[ \tilde{\lambda}_{\delta}(M) = \sup\{\ip{K_{\delta}f}{f}:f \in M \textrm{ and }\norm{f}=1 \} \]
and
\[ \tilde{\lambda}_{\delta,n} = \inf\{\tilde{\lambda}_{\delta}(M):M \subseteq \mathcal{C}_c^{\infty}(\delta,\beta-\delta) \textrm{ and dim}(M)=n\}.\]
Then, for $M \subseteq \mathcal{C}_c^{\infty}(\delta,\beta-\delta)$,
\begin{eqnarray*}
\tilde{\lambda}_{\delta}(M)	& = & \sup\{\ip{(-\Delta + c_{\delta}I)f}{f}:f \in M \textrm{ and }\norm{f}=1 \}\\
														& = & \sup\{\ip{(-\Delta f}{f}:f \in M \textrm{ and }\norm{f}=1 \} + c_{\delta}
\end{eqnarray*}
and hence
\[ \tilde{\lambda}_{\delta,n} = \frac{n^2\pi^2}{(\beta-2\delta)^2} + c_{\delta} \]
by the variational formula for the Dirichlet eigenvalues of $-\Delta$ on $(\delta,\beta-\delta)$. Since $\tilde{\lambda}_{\delta,n} \to +\infty$ as $n \to \infty$ and $\mathcal{C}_c^{\infty}(\delta,\beta-\delta)$ is a core for the quadratic form associated with $\tilde{K}_{\delta}$, $\tilde{K}_{\delta}$ has compact resolvent and its eigenvalues coincide with $\tilde{\lambda}_{\delta,n}$.
For all $M \subseteq \dom{H}$, define
\[ \mu(M) = \sup\{\ip{Hf}{f}:f \in M \textrm{ and }\norm{f}=1 \}. \]
If $f \in \mathcal{C}_c^{\infty}(\delta,\beta-\delta)$, then
\[ \ip{Hf}{f} = \ip{(-\Delta + V)f}{f} \leq \ip{(-\Delta + c_{\delta}I)f}{f} = \ip{K_{\delta}f}{f} \]
so $\mu(M) \leq \tilde{\lambda}_{\delta}(M)$ for all $M \subseteq \mathcal{C}_c^{\infty}(\delta,\beta-\delta)$.
Since $\mathcal{C}_c^{\infty}(0,\beta)$ is a core for $\hat{Q}$,
\begin{eqnarray*}
\mu_n	& = 		& \inf\{\mu(M):M \subseteq \mathcal{C}_c^{\infty}(0,\beta) \textrm{ and dim}(M)=n\}\\
			& \leq 	& \inf\{\mu(M):M \subseteq \mathcal{C}_c^{\infty}(\delta,\beta-\delta) \textrm{ and dim}(M)=n\}\\
			& \leq 	& \inf\{\tilde{\lambda}_{\delta}(M):M \subseteq \mathcal{C}_c^{\infty}(\delta,\beta-\delta) \textrm{ and dim}(M)=n\}\\
			& = 		& \frac{n^2\pi^2}{(\beta-2\delta)^2}+c_{\delta}
\end{eqnarray*}
for all $n \in \setN$. From the asymptotics of $V$, $c_{\delta}\delta^2 \to \max\{\frac{3}{4},\frac{1}{\varepsilon^2}-\frac{1}{4}\} =: c$ as $\delta \to 0$. Hence, given $\nu >0$, $c_{\delta} \leq (c+\nu)\delta^{-2}$ for all sufficiently small $\delta$. For such $\delta$,
\[\mu_n - \frac{n^2\pi^2}{\beta^2} \leq n^2\pi^2\left(\frac{1}{(\beta-2\delta)^2}-\frac{1}{\beta^2}\right) +\frac{c+\nu}{\delta^2} =: F_n(\delta)\]
for all $n \in \setN$. Hence, for all sufficiently large $n \in \setN$,
\begin{eqnarray*}
\mu_n - \frac{n^2\pi^2}{\beta^2}	& \leq 	& F_n(n^{-2/3})\\
																	& =  		& n^2\pi^2\left(\frac{1}{(\beta-2n^{-2/3})^2}-\frac{1}{\beta^2}\right) +(c+\nu)n^{4/3}\\
																	& =  		& n^2\pi^2\frac{4n^{-2/3}(\beta-n^{-2/3})}{\beta^2(\beta-2n^{-2/3})^2} +(c+\nu)n^{4/3}\\
																	& \sim 	& \left(\frac{4\pi^2}{\beta^3} +c +\nu\right)n^{4/3}
\end{eqnarray*}
as $n \to \infty$.
\qed
\end{proof}

\subsection{Eigenvalue numerics}\label{sec:numerics}

Let $\{a^{(m)}\}_{m=1}^{\infty}$ be a monotone decreasing sequence in $(0,1)$ converging to $0$ and $\{b^{(m)}\}_{m=1}^{\infty}$ a monotone increasing sequence in $(0,1)$ converging to $1$ such that $a^{(1)} < b^{(1)}$. For each $m \in \setN$, we consider the operator $L^{(m)}$ defined by \eqnref{eq:ldef} on the domain 
\[ \mathcal{D}^{(m)} = \{f \in \dom{\bar{L}}: f|_{I^{(m)}} \in \mathcal{C}^{\infty}(I^{(m)}) \textrm{ and } f(a^{(m)})=f(b^{(m)})= 0 \}.\]
Then $L^{(m)}$ is the operator corresponding to a regular Sturm-Liouville problem. By Theorem 14 of \cite{burkhill}, $L^{(m)}$ has infinitely many eigenvalues $\mu^{(m)}_n$ tending to $+\infty$.

\begin{lemma}\label{lem:decbound}
For each $n \in \setN$, $\{\mu^{(m)}_n\}_{m=1}^{\infty}$ is a decreasing sequence such that $\mu^{(m)}_n \geq \mu_n$ for all $m \in \setN$.
\end{lemma}
\begin{proof}
It is easy to see that
\[ \mathcal{D}^{(1)} \subseteq \mathcal{D}^{(2)} \subseteq \ldots \subseteq \dom{\bar{L}} \]
and hence the result follows immediately from the variational formula.
\qed
\end{proof}

\begin{lemma}\label{lem:conv}
For each $n \in \setN$, $\mu^{(m)}_n \to \mu_n$ as $m \to \infty$.
\end{lemma}
\begin{proof}
Since $\bar{L}$ has compact resolvent and $\mathcal{C}^{\infty}_c(0,1)$ is a form core for $\bar{L}$,
\[ \mu_n = \inf \{\mu(M):M \subseteq \mathcal{C}^{\infty}_c(0,1) \textrm{ and dim}(M) = n\}\]
where $\mu(M)$ is as in the variational formula. Hence, given $\eta>0$, there exists $M \subseteq \mathcal{C}^{\infty}_c(0,1)$ such that $\textrm{dim}(M)=n$ and $\mu_n \leq \mu(M) < \mu_n + \eta$. Since $M$ is finite-dimensional, we can choose $m_0 \in \setN$ such that $\supp{f} \in (a^{(m_0)},b^{(m_0)})$ for all $f \in M$. Hence $M \subseteq \mathcal{D}^{(m_0)}$ and so $\mu^{(m_0)}_n \leq \mu(M) < \mu_n + \eta$. Combining this with Lemma \ref{lem:decbound}, we have $\mu_n \leq \mu^{(m)}_n < \mu_n + \eta$ for all $m \geq m_0$.
\end{proof}

\begin{theorem}
If $\lambda_n$ is as in Corollary \ref{cor:evs} and $\lambda^{(m)}_n = \varepsilon \mu^{(m)}_n / 2$, then, for each $n \in \setN$, $\lambda^{(m)}_n$ decreases monotonically to $\lambda_n$ as $m \to \infty$.
\end{theorem}
\begin{proof}
Immediate from Lemmas \ref{lem:decbound} and \ref{lem:conv}.
\qed
\end{proof}

We have shown that we can obtain upper bounds on the eigenvalues of $L$ and $-iH$ by computing the eigenvalues of regular Sturm-Liouville problems, and moreover, that these bounds converge to the eigenvalues of these operators as we allow the endpoints of the regular problems converge to the endpoints of the Sturm-Liouville problem associated with $L$. We give tables of numerical calculations, obtained using the software package SLEDGE \cite{sledge} (obtainable from http://www.netlib.org/misc/sledge), of $\lambda_n^{(m)}$ for $a^{(m)} = 10^{-m}$, $b^{(m)} = 1 - 10^{-m}$ and various values of $\varepsilon$. For comparison, we also include numerical calculations from \cite{bobs}, \cite{davies-2007} and \cite{chugunova-2007}. Our computations are performed with an absolute error tolerance of $10^{-4}$, so we can be reasonably confident that any calculated values which exceed our upper bounds by more than this amount are not accurate to the stated precision. All values given in \cite{davies-2007} are consistent with these bounds when rounded to two decimal places, but there are some discrepancies at higher levels of precision. All values given in \cite{bobs} and \cite{chugunova-2007} are consistent with these bounds when rounded to one decimal place, but both sets of calculations have discrepancies at higher levels of precision.

\begin{table}[p]
\begin{tabular}{r|ccccccc}
$n$ & $\lambda^{(3)}_n$ & $\lambda^{(4)}_n$ & $\lambda^{(5)}_n$ & $\lambda^{(6)}_n$ &     $\lambda^{(7)}_n$ & $\lambda_n$ \cite{davies-2007} & $\lambda_n$ \cite{chugunova-2007}\\
\hline
  1	&  1.45457		&   1.44906		&   1.44851&   1.44845&   1.44844&	 1.4485&	 1.449323\\
  2	&  4.34574		&   4.31891		&   4.31614&   4.31587&   4.31584&	 4.3159&	 4.319645\\
  3	&  8.70318		&   8.63035		&   8.62264&   8.62186&   8.62178&	 8.6219&	 8.631474\\
  4	&  14.53324		&  14.38251		&  14.36590&  14.36421&  14.36405&	14.3638&	14.382886\\
  5	&  21.84048		&  21.57464		&  21.54473&  21.54167&  21.54137&	21.5414&	---\\
\end{tabular}
\caption{Numerical calculations for $\varepsilon=1$, compared with calculations from \cite{davies-2007} and \cite{chugunova-2007}}
\end{table}

\begin{table}[p]
\begin{tabular}{r|ccccccc}
$n$ & $\lambda^{(3)}_n$ &                $\lambda^{(4)}_n$ & $\lambda^{(5)}_n$ & $\lambda^{(6)}_n$ &                      $\lambda^{(7)}_n$ & $\lambda_n$ \cite{chugunova-2007}\\
\hline
  1&   1.17382&   1.16782&   1.16720&   1.16714&   1.16714&	 1.167342\\
  2&   2.99250&   2.97016&   2.96847&   2.96823&   2.96821&	 2.968852\\
  3&   5.54084&   5.48803&   5.48231&   5.48174&   5.48168&	 5.483680\\
  4&   8.82509&   8.72519&   8.71398&   8.71284&   8.71272&	 8.715534\\
  5&  12.85050&  12.68265&  12.66336&  12.66138&  12.66119&	---\\
  6&  17.61828&  17.36052&  17.32987&  17.32674&  17.32643&	---\\
  7&  23.13086&  22.75976&  22.71552&  22.71081&  22.71033&	---\\
  8&  29.39064&  28.88240&  28.81847&  28.81174&  28.81106&	---\\
  9&  36.39780&  35.72664&  35.63949&  35.63022&  35.62928&	---\\
 10&  44.15374&  43.29376&  43.17838&  43.16790&  43.16666&	---\\
\end{tabular}
\caption{Numerical calculations for $\varepsilon=0.5$, compared with calculations from \cite{chugunova-2007}}
\end{table}

\begin{table}[p]
\begin{tabular}{r|ccccccc}
$n$ &  $\lambda^{(3)}_n$ &                $\lambda^{(4)}_n$ & $\lambda^{(5)}_n$ & $\lambda^{(6)}_n$ &                      $\lambda^{(7)}_n$ &	$\lambda_n$ \cite{bobs} (numerical) & $\lambda_n$ \cite{davies-2007}\\
\hline
  1&   1.02908&   1.01149&   1.00961&   1.00942&   1.00940&	 1.0097&	 1.00968\\
  2&   2.11378&   2.07759&   2.07349&   2.07306&   2.07305&  2.0733&	 2.07334\\
  3&   3.29583&   3.23676&   3.22974&   3.22902&   3.22894&  3.2297&	 3.22978\\
  4&   4.59835&   4.51260&   4.50208&   4.50099&   4.50088&  4.5012&	 4.50134\\
  5&   6.03392&   5.91589&   5.90082&   5.89984&   5.89968&  5.8992&	 5.89993\\
  6&   7.60918&   7.45354&   7.43391&   7.43175&   7.43154&	 7.4298&	 7.43194\\
  7&   9.32789&   9.13017&   9.10350&   9.10063&   9.10034&	 9.0951&	 9.10097\\
  8&  11.19231&  10.94654&  10.91287&  10.90919&  10.90881& 10.8945&	10.9092\\
  9&  13.20382&  12.90464&  12.86256&  12.85789&  12.85742&	12.8252&	12.8578\\
 10&  15.36360&  15.00536&  14.95367&  14.94786&  14.94727&	14.8820&	14.9478\\
\end{tabular}

\caption{Numerical calculations for $\varepsilon=0.1$, compared with calculations from \cite{bobs} and \cite{davies-2007}}
\end{table}

\appendix
\section{Asymptotics of $V$}\label{app:asymptotics}

We want to analyse the asymptotic behaviour of
\begin{equation}\label{eq:v}
V(s) = -c(s)c''(s)\frac{p(\phi(s))}{\phi'(s)} - c(s)c'(s)\frac{\phi'(s)^2 p'(\phi(s)) - p(\phi(s))\phi''(s)}{\phi'(s)^2}
\end{equation}
as $s \to 0+$ and $s \to \beta-$. We have
\begin{eqnarray}
c		& = &	k|_{(0,1)} \circ \phi \label{eq:c}\\
c'	& = &	(k'|_{(0,1)} \circ \phi) \phi' \label{eq:cdash}\\
c''	& = &	(k''|_{(0,1)} \circ \phi) \phi'^2 + (k'|_{(0,1)} \circ \phi) \phi'' \label{eq:cddash}
\end{eqnarray}
where $k$ is the analytic function defined by
\[ k(z) = z^{1/4} (1-z)^{-1/2\varepsilon-1/4} (1+z)^{1/2\varepsilon-1/4} \]
for all $z \in \setC \setminus ((-\infty,0] \cup [1,+\infty))$.

\begin{lemma}\label{lem:kasym}
Asymptotically,
\begin{eqnarray*}
k(z)	& \sim & 	z^{1/4}\\
k'(z)	& \sim & 	\frac{1}{4}z^{-3/4}\\
k''(z)& \sim &	-\frac{3}{16}z^{-7/4}
\end{eqnarray*}
as $z \to 0$ and
\begin{eqnarray*}
k(z)	& \sim & 	2^{1/2\varepsilon-1/4}(1-z)^{-1/2\varepsilon-1/4}\\
k'(z)	& \sim & 	\left(\frac{1}{2\varepsilon}+\frac{1}{4}\right)2^{1/2\varepsilon-1/4}(1-z)^{-1/2\varepsilon-5/4}\\
k''(z)& \sim &	\left(\frac{1}{2\varepsilon}+\frac{1}{4}\right)\left(\frac{1}{2\varepsilon}+\frac{5}{4}\right)
2^{1/2\varepsilon-1/4}(1-z)^{-1/2\varepsilon-9/4}
\end{eqnarray*}
as $z \to 1$.
\end{lemma}
\begin{proof}
We may write
\[ k(z) = z^{1/4} k_0(z)\]
where
\begin{equation}\label{eq:knought}
k_0(z) = (1-z)^{-1/2\varepsilon-1/4} (1+z)^{1/2\varepsilon-1/4}
\end{equation}
is analytic in a neighbourhood of $0$. The function $k_0$ has a power series expansion
\[ k_0(z) = \sum_{n=0}^{\infty} a_n z^n \]
valid in some disc $D(0;R_0)$. Putting $z = 0$ in \eqnref{eq:knought}, we see that $a_0 = 1$. We thus have
\begin{eqnarray*}
k(z)		& = & z^{1/4} + \sum_{n=1}^{\infty} a_n z^{n+1/4}\\
k'(z)		& = & \frac{1}{4}z^{-3/4} + \sum_{n=1}^{\infty} \left(n+\frac{1}{4}\right) a_n z^{n-3/4}\\
k''(z)	& = & -\frac{3}{16}z^{-7/4} + \sum_{n=1}^{\infty} \left(n+\frac{1}{4}\right)\left(n-\frac{3}{4}\right) a_n z^{n-7/4}
\end{eqnarray*}
in the cut disc $D(0;R_0) \setminus (-R_0,0]$, and the stated asymptotics as $z \to 0$ follow.

Similarly, by writing
\[ k(z) = (1-z)^{-1/2\varepsilon-1/4} k_1(z) \]
where
\[ k_1(z) = z^{1/4} (1+z)^{1/2\varepsilon-1/4} \]
is analytic in a neighbourhood of $1$, we obtain the stated asymptotics as $z \to 1$.
\qed
\end{proof}

\begin{lemma}\label{lem:phiasym}
Asymptotically,
\begin{eqnarray*}
\phi(s)		& \sim & 2^{-2} s^2\\
\phi'(s)	& \sim &	 2^{-1}s\\
\phi''(s)	& \sim &	2^{-1}
\end{eqnarray*}
as $s \to 0+$ and
\begin{eqnarray*}
1-\phi(s)	& \sim & 2^{-1} (\beta - s)^2\\
\phi'(s)	& \sim &	\beta - s\\
\phi''(s)	& \sim &	-1
\end{eqnarray*}
as $s \to \beta-$.
\end{lemma}
\begin{proof}
The inverse $\psi$ of $\phi$ is given by
\[\psi(t) = \int_0^t w(y)^{1/2}p(y)^{-1/2} \mathrm{d}y \]
for all $t \in [0,1]$. Since
\[ w(y)^{1/2}p(y)^{-1/2} = y^{-1/2}(1-y)^{-1/2}(1+y)^{-1/2} \]
for all $y \in (0,1)$,
\[ \psi(t) \sim 2 t^{1/2} \]
as $t \to 0+$ and
\[ \beta-\psi(t) = \int_t^1 y^{-1/2}(1-y)^{-1/2}(1+y)^{-1/2} \mathrm{d}y \sim 2^{1/2}(1-t)^{1/2} \]
as $t \to 1-$. Putting $t = \phi(s)$, we obtain the stated asymptotics for $\phi$.

The stated asymptotics for $\phi'$ follow immediately from the asymptotics of $\phi$ and the fact that
\begin{eqnarray}
\phi'(s) 	& = & \frac{1}{\psi'(\phi(s))}\\
					& = & p(\phi(s))^{1/2}w(\phi(s))^{-1/2}\\
					& = & \phi(s)^{1/2}(1-\phi(s))^{1/2}(1+\phi(s))^{1/2} \label{eq:phidash}
\end{eqnarray}
for all $s \in (0,\beta)$. Differentiating \eqnref{eq:phidash} and using the asymptotics we have calculated for $\phi$ and $\phi'$, we obtain the stated asymptotics for $\phi''$.
\qed
\end{proof}

\begin{corollary}\label{cor:casym}
Asymptotically,
\begin{eqnarray*}
c(s)	& \sim &	2^{-1/2}s^{1/2}\\
c'(s)	& \sim &	2^{-3/2}s^{-1/2}\\
c''(s)& \sim &	-2^{-5/2}s^{-3/2}
\end{eqnarray*}
as $s \to 0+$ and
\begin{eqnarray*}
c(s)	& \sim & 2^{1/\varepsilon} (\beta-s)^{-1/\varepsilon-1/2}\\
c'(s) & \sim & \left(\frac{1}{\varepsilon}+\frac{1}{2}\right)2^{1/\varepsilon} (\beta-s)^{-1/\varepsilon-3/2}\\
c''(s)& \sim & \left(\frac{1}{\varepsilon}+\frac{1}{2}\right)\left(\frac{1}{\varepsilon}+\frac{3}{2}\right)2^{1/\varepsilon} (\beta-s)^{-1/\varepsilon-5/2}
\end{eqnarray*}
as $s \to \beta-$.
\end{corollary}
\begin{proof}
This follows immediately from equations \eqnref{eq:c}, \eqnref{eq:cdash}, \eqnref{eq:cddash} and Lemmas \ref{lem:kasym} and \ref{lem:phiasym}. \qed
\end{proof}

\begin{corollary}\label{cor:pasym}
Asymptotically,
\begin{eqnarray*}
p(\phi(s))	& \sim & 	1\\
p'(\phi(s))	& \sim &	-\frac{2}{\varepsilon}
\end{eqnarray*}
as $s \to 0+$ and
\begin{eqnarray*}
p(\phi(s))	& \sim & 	2^{-2/\varepsilon} (\beta-s)^{2+2/\varepsilon} \\
p'(\phi(s))	& \sim &	-\left(1+\frac{1}{\varepsilon}\right) 2^{1-2/\varepsilon} (\beta-s)^{2/\varepsilon}
\end{eqnarray*}
as $s \to \beta-$.
\end{corollary}
\begin{proof}
The asymptotics of $p$ and $p'$ are:
\begin{eqnarray*}
p(x)	& \sim & 1\\
p'(x)	& \sim & -\frac{2}{\varepsilon}
\end{eqnarray*}
as $x \to 0$ and
\begin{eqnarray*}
p(x)	& \sim & 2^{1-1/\varepsilon} (1-x)^{1+1/\varepsilon}\\
p'(x)	& \sim & -\left(1+\frac{1}{\varepsilon}\right) 2^{1-1/\varepsilon} (1-x)^{1+1/\varepsilon}
\end{eqnarray*}
as $x \to 1$. The stated asymptotics follow from this and Lemma \ref{lem:phiasym}. \qed
\end{proof}

\begin{theorem}
Asymptotically,
\[ V(s) \sim \frac{3}{4} s^{-2} \]
as $s \to 0+$ and
\[ V(s) \sim \left(\frac{1}{\varepsilon^2}-\frac{1}{4}\right) (\beta-s)^{-2} \]
as $s \to \beta-$.
\end{theorem}
\begin{proof}
This follows from \eqnref{eq:v}, Lemma \ref{lem:phiasym} and Corollaries \ref{cor:casym} and \ref{cor:pasym}.
\qed
\end{proof}

\bibliographystyle{unsrt}
\bibliography{D:/John/LaTeX/References}

\end{document}